\documentclass[12pt]{amsart}
\usepackage{amssymb}

\usepackage{hyperref}
\hypersetup{colorlinks,citecolor=blue,linkcolor=blue}

\newcommand\SL{\operatorname{SL}}
\newcommand\GL{\operatorname{GL}}
\newcommand\SO{\operatorname{SO}}

\newcommand\N{{\mathbb N}}
\newcommand\Z{{\mathbb Z}}
\newcommand\Q{{\mathbb Q}}
\newcommand\R{{\mathbb R}}
\newcommand\C{{\mathbb C}}
\newcommand\A{{\mathbb A}}
\newcommand\F{{\mathbb F}}
\newcommand\cO{{\mathcal O}}

\newcommand\E{{\mathcal E}}
\newcommand\D{\mathcal{D}}

\newcommand\qsG{\mathcal{G}}
\newcommand\G{{\mathrm G}}
\newcommand\adG{{\overline{\G}}}
\newcommand\Pa{{\mathrm P}}
\newcommand\cP{{\mathcal P}}

\newcommand\B{{\mathcal B}}

\newcommand\Mv{{\overline{\mathrm M}_v}}
\newcommand\cMv{{\overline{\mathcal M}_v}}
\newcommand\Mvo{{\overline{\mathrm M}_{v_0}}}
\newcommand\cMvo{{\overline{\mathcal M}_{v_0}}}

\newcommand\Kv{K_v^o}
\newcommand\cK{{\mathcal K}}
\newcommand\oG{{\overline{\Lambda}}}
\newcommand\GI{{\Lambda(I)}}
\newcommand\ocP{{\overline{\cP}}}

\newcommand\ru{\mathrm{L}_H(x)}
\newcommand\muu{m_H(x)}

\newcommand\Ker{\mathrm{Ker}}
\newcommand\Comm{\mathrm{Comm}_H\:}
\newcommand\dm{\mathrm{dim}}
\newcommand\rk{\mathrm{rk}}

\newcommand\An{\mathrm{A}}
\newcommand\Bn{\mathrm{B}}
\newcommand\Cn{\mathrm{C}}
\newcommand\Dn{\mathrm{D}}
\newcommand\En{\mathrm{E}}
\newcommand\Fn{\mathrm{F_4}}
\newcommand\Gn{\mathrm{G_2}}

\newtheorem{theorem}{Theorem}
\newtheorem{problem}{Problem}

\newtheorem{prop}{Proposition}[section]
\newtheorem{lemma}[prop]{Lemma}
\newtheorem{thm}[prop]{Theorem}
\newtheorem{cor}[prop]{Corollary}

\newtheorem{conj}[prop]{Conjecture}
\newtheorem{claim}[prop]{Claim}

\theoremstyle{definition}
\newtheorem{rem}[prop]{Remark}

\begin{document}

\title{Manifolds counting and class field towers}

\dedicatory{Dedicated to the memory of A. I. Fet}

\author{Mikhail Belolipetsky}
\thanks{Belolipetsky partially supported by EPSRC grant EP/F022662/1}
\address{
IMPA\\
Estrada Dona Castorina, 110\\
22460-320 Rio de Janeiro, Brazil}
\email{mbel@impa.br}

\author{Alexander Lubotzky}
\thanks{Lubotzky partially supported by BSF (US -- Israel), NSF and ERC}
\address{
Institute of Mathematics\\
Hebrew University\\
Jerusalem 91904, Israel}
\email{alexlub@math.huji.ac.il}


\date{\today}

\begin{abstract}
In \cite{BGLM} and \cite{GLNP} it was conjectured that if $H$ is a simple Lie group of real rank at least $2$, then the number of conjugacy classes of (arithmetic) lattices in $H$ of covolume at most $x$ is $x^{(\gamma(H)+o(1))\log x/\log\log x}$ where $\gamma(H)$ is an explicit constant computable from the (absolute) root system of $H$. In this paper we prove that this conjecture is false. In fact, we show that the growth is at rate $x^{c\log x}$. A crucial ingredient of the proof is the existence of towers of field extensions with bounded root discriminant which follows from the seminal work of Golod and Shafarevich on class field towers.
\end{abstract}

\maketitle

\section{Introduction}

Let $H$ be a non-compact simple Lie group endowed with a fixed Haar measure
$\mu$, $K$~a maximal compact subgroup of $H$ and $X = H/K$ the associated
symmetric space. A classical theorem of Wang \cite{Wa} asserts that if $H$ is
not locally isomorphic to $\SL_2(\R)$ or $\SL_2(\C)$, then for every $0< x
\in\R$ there exist only finitely many Riemannian orbifolds covered by $X$ with
volume at most $x$. Consequently, if $\mathrm{L}_H(x)$ (resp. $\mathrm{TFL}_H(x)$,
$\mathrm{AL}_H(x)$) denotes the number of conjugacy
classes of lattices (resp. torsion-free lattices, arithmetic lattices) in
$H$ of covolume at most $x$, then $\mathrm{L}_H(x)$ is finite for every $x$.
For $\mathrm{AL}_H(x)$ this is also true even for $H = \SL_2(\R)$ or
$\SL_2(\C)$ by a result of Borel \cite{Bo2}.

In recent years there has been a growing interest in the asymptotic
behavior of these functions (cf. \cite{BGLM}, \cite{G1}, \cite{G2}, \cite{GLNP},
\cite{B} and \cite{BGLS}). Super-exponential upper bounds were given in many
cases, and at least for rank one groups $\SO(n,1)$ these bounds are optimal.

The current paper is devoted to the study of $\mathrm{L}_H(x)$ for groups $H$ with
real rank at least $2$. Here one expects a slower rate
of growth: Recall that in this case, by Margulis's arithmeticity
theorem (see \cite{Ma}), every lattice $\Gamma$ in $H$ is {\em arithmetic}, i.e. there exists
a number field $k$ with ring of integers $\cO$ and the set of archimedean
valuations $V_\infty$, an absolutely simple, simply connected $k$-group $\G$
and an epimorphism $\phi: G = \prod_{v\in V_\infty} \G(k_v) \to H$, such that
$\Ker(\phi)$ is compact and $\phi(\G(\cO))$ is commensurable with $\Gamma$.
Thus for groups $H$ of real rank at least $2$, we have $\mathrm{L}_H(x)=\mathrm{AL}_H(x)$.
Moreover, Serre conjectured (\cite{S1}) that for all lattices $\Gamma$ in such $H$,
$\Gamma$ has the {\em congruence subgroup property} (CSP), i.e.
$\Ker(\widehat{\G(\cO)} \to \G(\widehat{\cO}))$ is finite in the notations above.
Assuming the conjecture, the question of counting lattices in $H$ boils down to counting
arithmetic groups and their congruence subgroups.
A related conjecture which is also relevant for us is
{\em Margulis-Platonov} (MP) {\em conjecture} (cf.~\cite{PR10}). It says that all normal
subgroups of $\G(k)$ are of standard form coming from the non-archimedean
valuations of $k$ with respect to which $\G(k_v)$ is anisotropic (in
particular, $\G(k)$ does not have any noncentral proper normal subgroups if $\G$ is
$k_v$-isotropic for all $v$).

The conjecture of Serre is proved by now for all non-uniform lattices and for ``most''
of the uniform ones, excluding certain cases when $H$ is of type $\An_n$, $\Dn_4$ or
$\En_6$, and the same is also true for MP (see \cite[Chapt.~9]{PR} and \cite{PR10}
for the details and precise statements).
Moreover, very precise estimates for the number of congruence subgroups in a given
lattice are obtained in \cite{Lu}, \cite{GLP} and \cite{LN}, some of these are conditional
on the validity of the generalized Riemann hypothesis (GRH, cf. \cite{We}).
These results led to a conjecture in \cite{BGLM} that for groups $H$ of $\R$-rank$\ge2$,
$\mathrm{L}_H(x)$ grows like $x^{c\log x/\log\log x}$.
In fact, a more precise conjecture is made in \cite{GLNP}, where it is suggested that
\begin{equation}\label{eq_11}
\lim_{x\to\infty} \frac{\log \mathrm{L}_H(x)}{(\log x)^2 / \log\log x} = \gamma(H),\quad
{\text with\ \ } \gamma(H) = \frac{(\sqrt{h(h+2)}-h)^2}{4h^2},
\end{equation}
where $h$ is the Coxeter number of the (absolute) root system corresponding to $H$
(i.e. the root system of the split form of $H$).

\medskip

In this paper we prove that the conjecture is false! The correct rate of growth
is $x^{\log x}$. It is still possible to show that the conjecture is essentially true
if one restricts to {\it non-uniform} lattices for which we refer to  \cite{BL2}.

\begin{theorem} \label{theorem}
Let $H$ be a simple Lie group of real rank at least $2$. Then
\begin{itemize}
\item[(i)] There exists a positive constant $a$ such that $\ru \ge x^{a\log
x}$ for all sufficiently large $x$.
\item[(ii)] Assuming the CSP and MP, there exists a positive constant $b$
such that $\ru \le x^{b\log x}$ for all sufficiently large $x$.
\end{itemize}
\end{theorem}

A crucial ingredient in the proof of part (i) of the theorem is the existence
of infinite class field towers of totally real fields as established by Golod and
Shafarevich~\cite{GS}. In \S\ref{section:NT}, we elaborate on it using the theory of Pisot
numbers to get also sequences of fields of arbitrarily large degree but with a fixed number
of complex places and a bounded root discriminant. Using these fields
we construct a sequence of arithmetic lattices in $H$ with covolume going to
infinity and with a particularly large number of subgroups of small index.
See below for details. It is interesting to mention that arithmetic
lattices with fields of definition of growing degrees also come out
in a connection with the Lehmer conjecture --- see \cite{B2}.

\medskip

Our argument gives explicit estimates for the constants $a$ and $b$ in Theorem \ref{theorem} but falls short from answering the following:
\begin{problem}\label{problem}
Does $\displaystyle{\lim_{x\to\infty} \frac{\log \mathrm{L}_H(x)}{(\log x)^2}}$ exist? And if so, what is its value?
\end{problem}

\medskip

The proof of the theorem uses methods developed in \cite{B},
\cite{GLP} and \cite{LN} (see also \cite{GLNP}). We would like to point out
that unlike some of the results in \cite{GLP} and \cite{LN}, Theorem \ref{theorem} does not
depend on the GRH.

As all this suggests, our work is actually about counting arithmetic
lattices and their congruence subgroups. It is therefore not surprising that it
eventually boils down to various number theoretic problems. The wealth and
diversity of number theoretic ingredients involved in proving Theorem \ref{theorem}
and results of \cite{BL2} is exciting and may suggest some topics for future
study.

\medskip

Before describing the method of proof, let us put our main result in a
more general perspective.
In \cite{BGLM} the rate of growth of $\mathrm{TFL}_H(x)$ was
determined for $H = \SO(n,1)$, $n\ge4$; it is super-exponential. The lower
bound there is already obtained by considering a suitable fixed lattice in
$\SO(n,1)$ and its finite index subgroups. The upper bound is proved by
geometric methods. These geometric methods were extended in \cite{G1} and
\cite{G2} to more general semisimple groups.
In \cite{BGLS} a very precise super-exponential estimate for $\mathrm{AL}_H(x)$
is given for $H = \SL_2(\R)$. There again the full rate of growth is already
obtained by considering the finite index subgroups of a single lattice. Moreover,
in \cite{GLP} and \cite{LN} (see also \cite{GLNP})
precise asymptotic estimates were given for the growth rate of the number
of congruence subgroups in a fixed lattice $\Lambda$ in $H$. (Some of the
results there are conditional on the GRH). The rate of growth turns out to depend
only on $H$ and not on $\Lambda$. All this
suggested that the rate of growth of the finite index subgroups within one
lattice is the main contribution to $\mathrm{L}_H(x)$. This led to the conjecture
mentioned above. Moreover, in \cite{B} it is shown that the growth rate of the
{\it maximal} arithmetic lattices in $H$ is very small.
This provided more evidence in favor of the conjecture.
Recently A.~Salehi Golsefidy~\cite{Sa} showed that indeed in simple Lie
groups over local fields of positive characteristic, the total growth
of lattices is of the same growth type as the subgroup growth of a
single lattice.

In \cite{BL2} we will show that
the conjecture is essentially true for non-uniform lattices but Theorem \ref{theorem}
here shows, somewhat surprisingly, that it is not true in general.
In fact, we discover here a new phenomenon: the main contribution
to the growth of uniform lattices in $H$ does not come from subgroups of a
single lattice. As it will be explained below, it comes from a ``diagonal
counting'' when we run through different arithmetic groups $\Gamma_i$ defined
over number fields $k_i$ of different degrees $d_i$, and for each $\Gamma_i$ we count
some of its subgroups. The difference between the uniform and non-uniform cases relies
on the fact that all non-uniform lattices in $H$ are defined over number fields of
a bounded degree over $\Q$. On the other hand, uniform lattices may come from
number fields $k_i$ of arbitrarily large degrees, i.e., $d_i\to\infty$.

\medskip

We now briefly sketch the main argument. If $\Gamma$ is an arithmetic
lattice obtained as $\phi(\G(\cO))$ defined above, then there is an explicit
formula \cite{Pr} for its covolume in $H$. The analysis of this formula and also
the growth of the low-index congruence subgroups of $\phi(\G(\cO))$ shows that we
can expect fast subgroup growth if we consider groups over fields of growing
degree with relatively slow growing discriminant $\D_k$. More precisely, we can
combine this two entities together into the so-called root-discriminant $rd_k =
\D_k^{1/{\rm deg}\: k}$ and then look for a sequence of number fields $k_i$ with
degrees growing to infinity but with bounded $rd_{k_i}$. In a seminal work Golod
and Shafarevich \cite{GS} came up with a construction of infinite class field towers.
It is such a tower of number fields $k_i$ that we use to define our arithmetic
subgroups $\Gamma_i$. Galois cohomology methods show the existence of suitable
$k_i$-algebraic groups $\G_i$ which give rise to arithmetic lattices $\Gamma_i
= \G_i(\cO_i)$ in $H$ whose covolume is bounded exponentially in $d_i = {\rm
deg}\: k_i$. We then present $c^{d_i^2}$ congruence subgroups of $\Gamma_i$
whose covolume is still bounded exponentially in $d_i$. Using the theory of
Bruhat-Tits buildings in \S\ref{section:covers vs manifolds} we show that
sufficiently many of such congruence subgroups are not conjugate to each other
in $H$. This will complete the proof of the lower bound of Theorem \ref{theorem},
at least for most real simple Lie groups $H$. The remaining cases require further
consideration: for example, if $H$ is a complex Lie group, the fields $k_i$ should
be replaced by suitable extensions obtained via the help of the theory of
Pisot numbers. These fields do not form a class field tower any more but still have
bounded root discriminant.

The proof of the upper bound presents a new type of difficulty: this time we need
to obtain a uniform upper bound on growth which does not depend on the degrees of
the defining fields. (This is what makes the growth rate $x^{\log x}$ instead of
$x^{\log x/\log\log x}$.) A key ingredient of the proof is an important theorem of
Babai, Cameron and P\'alfy (see Theorem \ref{thm:BCP}) which bounds the size of
permutation groups with restricted Jordan-Holder components. This theorem was
previously used in \cite{Lu} to study the subgroup growth of lattices
defined over global fields of positive characteristic. Bringing related technique
to the number field case presents certain challenges and requires developing
some new ``subgroup growth'' methods.
We refer to \S\ref{section:upper bounds} for the details of the argument.

\medskip

The paper is organized as follows. After introducing some notations and
conventions in \S\ref{section:notations}, we supply in \S\ref{section:NT}
the needed number theoretic background:
we quote the Golod-Shafarevich work and use it with the theory of Pisot
numbers to get families of number fields with bounded root discriminant and
a given number of complex embeddings. In \S\ref{section:arithmetic} we analyze
carefully Prasad's formula for the covolume of arithmetic lattices.
In \S\ref{section:covers vs manifolds} we tackle the subtle
difference between counting covers of a given manifold $M$ (which is what we
get by counting finite index subgroups of $\pi_1(M)$) and counting
manifolds covering $M$~--- which is what is relevant in the current paper.
This issue often occurs in geometric considerations, for example, in constructions
of manifolds which are isospectral but not isometric. We develop the required technique
only up to the point needed in the current paper and leave some questions for
further research. The theory of Bruhat-Tits building and their combinatorial
growth plays a major role here. In \S\ref{sec:thm1 lower} we prove the lower
bound of Theorem \ref{theorem} using the results in \S\S \ref{section:NT},
\ref{section:arithmetic} and \ref{section:covers vs manifolds}, while in
\S\ref{section:upper bounds} we prove the upper bound. Finally, in
\S\ref{section:semisimple} we extend the theorem to semisimple Lie groups.

\medskip

\noindent
{\em Acknowledgements.} The authors are grateful to T.~Gelander, G.~Prasad, A.~Rapinchuk, 
A.~Salehi Golsefidy and Ya.~Varshavsky for helpful discussions.

\section{Notations and conventions}\label{section:notations}


Let $H$ be a semisimple Lie group without compact factors.
Its subgroup $\Gamma$ is called a {\em lattice}
if $\Gamma$ is discrete in $H$ and its covolume (with respect to some and hence
any Haar measure on $H$) is finite. A lattice is called {\em irreducible} if $\Gamma N$
is dense in $H$, for every noncompact, closed, normal subgroup $N$ of $H$.
A lattice is called {\em uniform} (resp.
{\em non-uniform}) if $H/\Gamma$ is compact (resp. non-compact).

Two groups $\Gamma_1$ and $\Gamma_2$ are called {\it commensurable} if
$\Gamma_1\cap\Gamma_2$ is of finite index in both of them. If $\Gamma$ is a
lattice in $H$, its {\it commensurability subgroup} (or {\it commensurator}) in
$H$ is defined as
\begin{equation*}
{\rm Comm}_H(\Gamma) = \{ g\in H \mid g^{-1}\Gamma g\ {\rm \ \and\ } \Gamma
{\rm\ are\ commensurable}\}.
\end{equation*}

For a group (resp. profinite group) $G$ we define its {\em rank} $\rk(G)$ as the supremum of the minimal number of generators over the finitely generated subgroups (resp. open subgroups) of $\G$. If $G$ is a finite group, its {\em $p$-rank} is defined by $\rk_p(G) = \rk(P)$, where $p$ is a prime and $P$ is a Sylow $p$-subgroup of $G$.

\medskip

Along the lines we shall often come to arithmetic considerations, for which we
now fix some notations. Throughout this paper $k$ will always denote a number
field, $\cO = \cO_k$ is its ring of integers and $\D_k$ is the absolute value
of the discriminant $\Delta_k$. The set of valuations
(places) of $k$, $V = V(k)$, is the union of the set $V_\infty$ of archimedean
and the set $V_f$ of nonarchimedean (finite) places of $k$.
The number of archimedean places of $k$ is denoted by $a = a_k = \# V_\infty$,
and $r_1$, $r_2$ denote the number of real and complex places of $k$, respectively
(so $a = r_1 + r_2$ and $d = d_k = [k:\Q] = r_1 + 2r_2$).
Given a nonarchimedean place $v\in V_f$, the
completion of $k$ with respect to $v$ is a nonarchimedean local field $k_v$, its residue
field, which will be denoted by $\F_v$ or $\F_q$, is a finite field of cardinality
$q = q_v$. Finally, $\A = \A(k) = \prod_{v\in V}'k_v$ is the ring of ad\`eles of $k$,
where $\prod'$ denotes a restricted product.

\medskip

All logarithms in this paper will be taken to base $2$. For a real number $x$, $[x]$ denotes the largest integer $\le x$. The number of elements of a finite set $S$ will be denoted by $\# S$, while the order of a finite group $G$ will be denoted by $|G|$.

\medskip

Whenever it is not stated otherwise, the constants $c_1$, $c_2$ and etc. depend only on the Lie group $H$.

\section{Number-theoretic background}\label{section:NT}

\subsection{} Let $\alpha_1$, \dots, $\alpha_d$ be a $\Z$-basis of $\cO_k$ (i.e. an integral basis of $k$), and let $v_1$, \dots, $v_d$ denote the archimedean embeddings of $k$. By definition, the {\em discriminant} $\Delta_k = {\rm det}\bigl[v_j(\alpha_i)\bigr]^2$ and $\D_k$ is its absolute value. The discriminant is related to the volume of the fundamental domain of the integral lattice in $k$. As we will see later on, this relation goes further to the covolumes of arithmetic lattices in semisimple Lie groups. We will also use a notion of {\em root discriminant} of $k$ which is defined by $rd_k = \D_k^{1/d}$, where $d = [k:\Q]$.

Let us recall the following well known results.

\begin{thm}\label{thm:Minkowski} {\rm (Minkowski, see \cite[Theorem 4, p. 119]{La})}
Let  $k$ be a number field of degree $d = r_1 + 2r_2$. There exists a nonzero $\alpha\in\cO_k$ whose norm satisfies
\begin{equation*}
|N(\alpha)| \le \left(\frac4\pi\right)^{r_2}\frac{d!}{d^d}\sqrt{\D_k}.
\end{equation*}
\end{thm}
The proof of this theorem follows from the existence of lattice points in convex bodies in $\R^d$ whose volume is big enough relative to a fundamental region for the lattice.

By Stirling's formula, $d! = \sqrt{2\pi d}\left(\frac{d}{e}\right)^d e^{\theta/12d}$ with $0< \theta < 1$ (see \cite[p. 122]{La}). This, together with the fact that $|N(\alpha)| \ge 1$ for $0\neq\alpha\in\cO_k$, allows us to deduce that $\D_k > \left(\frac\pi4\right)^{2r_2}\frac{1}{2\pi d} e^{2d-(1/6d)}$. We shall often use the following form of this estimate:
\begin{cor}\label{cor:Minkowski} {\rm (see \cite[Theorem 5, p. 121]{La})}
There exists an absolute constant $C>0$ such that for any $k\neq\Q$, $d \le C\log\D_k$.
\end{cor}

\subsection{} Let us call a sequence of pairwise non-isomorphic fields $(k_i)_{i\in\N}$ {\em asymptotically bounded} if there exists a constant $c_0$ such that for every $i$, the root discriminant $rd_{k_i} \le c_0$. The definition implies that the degree of the fields in an asymptotically bounded sequence goes to infinity (it can be deduced from Minkowski's theorem that the number of fields with bounded $\D_k$ is finite, hence the number of fields with bounded root discriminant and bounded degree is also finite). The existence of asymptotically bounded sequences is not obvious, it follows from the work of Golod and Sha\-fa\-re\-vich on the class field towers.

\begin{thm}\label{thm:GS}{\rm (Golod-Shafarevich \cite{GS})}
There exists an infinite tower of unramified extensions of a totally real number field $k$.
\end{thm}

Given an unramified extension $l/k$, we have $\D_l = \D_k^{[l:k]}$ (by \cite[Prop. 8, p. 62 and Prop. 14, p. 66]{La}), and thus $rd_l = rd_k$. Therefore the root discriminant is constant along a tower of unramified extensions, which implies that such towers are asymptotically bounded.  A well known explicit sequence of totally real fields which satisfy
Golod-Shafarevich criterion was constructed by Martinet in \cite{Martinet}, the degrees of the fields are powers of $2$ and $c_0 = rd_{k_i} = 1058.565...$\,. A question about the smallest possible value of $c_0$ is important for various applications and is still open. It is known that a smaller constant can be achieved if we do not require the extensions to be unramified. The best current result in this direction is obtained by Hajir and Maire in \cite{HM}, it provides an asymptotically bounded sequence of totally real fields with $c_0 = 954.3...$\,.

\subsection{}
Our next goal is to construct asymptotically bounded sequences of fields which have a fixed nonzero number of complex places. Note that the results mentioned above do not apply to this case, as in an unramified tower the number of complex places is either zero or grows with the degree (the same applies also to tamely ramified towers in \cite{HM}). In order to deal with this problem we use some results about Pisot numbers.

Assume that the field $k$ has at least one real place. The number $\theta\in k$ is called a {\em Pisot number} (or Pisot-Vijayaraghavan number) if for a real place $v_1: k\to\R$ we have $v_1(\theta) > 1$ and for all other $v_j\in V_\infty$, $|v_j(\theta)| < 1$.

\begin{lemma}\label{lemma:Pisot}
Let $k$ be a totally real number field of degree $d$.
\begin{itemize}
\item[(a)]
There exists a Pisot number $\theta\in k$ such that $\theta$ has degree $d$ and $|N(1-\theta)| < \D_k^\delta$ for some absolute constant $\delta$.
\item[(b)]
Moreover, for any $t$ such that $1\le t \le d$ there exist $t$ different Pisot numbers $\theta_1$, \dots, $\theta_t$ satisfying the conditions of part (a) and such that $\alpha = (1-\theta_1)\ldots(1-\theta_t)$ is negative at $t$ archimedean places of $k$ and positive at the remaining $d-t$ places.
\end{itemize}
\end{lemma}

\begin{proof} (a) It is well known that there exist Pisot numbers $\theta\in k$ which generate $k$ over $\Q$ (see e.g. \cite[Theorem 5.2.2, p. 85]{BDGPS}), thus it remains to show that we can choose such $\theta$ that the upper bound for the norm of $1-\theta$ holds. In order to do so we need to recall the proof of the existence of $\theta$: The argument uses Minkowski's theorem and implies that we can choose $\theta$ such that $1 < v_1(\theta) \le 2^{d-1}\sqrt{\D_k}$ and $|v_j(\theta)| \le 1/2$ for $v_j\in V_\infty\smallsetminus\{v_1\}$ (see loc. cit. for the details). Now if $P(x)$ is the minimal polynomial of $\theta$, then $|N(1-\theta)| = |P(1)|$. We have
$$
|P(1)| = |(1-v_1(\theta))\cdot\ldots\cdot(1-v_d(\theta))| \le (2^{d-1}\sqrt{\D_k} + 1)\left(\frac32\right)^{d-1} = 3^{d-1}\sqrt{\D_k} + \left(\frac32\right)^{d-1}.
$$
By Corollary \ref{cor:Minkowski}, the degree $d$ is bounded by $C\log\D_k$, hence we obtain $|N(1-\theta)|\le\D_k^\delta$ where $\delta$ depends only on $C$.

(b) Using part (a) we can find $t$ different Pisot numbers $\theta_1,\dots,\theta_t \in k$ such that $v_i(\theta_i) > 1$, $|v_j(\theta_i)| < 1$ for $j\neq i$ and $|N(1-\theta_i)| < \D_k^\delta$ ($1\le i \le t$, $v_j$ are the infinite places of $k$). It follows that $\alpha = (1-\theta_1)\ldots(1-\theta_t)$ satisfies the conditions at the infinite places.
\end{proof}

\begin{cor} \label{cor:Pisot}
Given $t\in\N$, there exists an asymptotically bounded sequence of fields $(l_i)_{i\in\N}$ such that $r_2(l_i) = t$ for all $i$.
\end{cor}

\begin{proof} We start with an infinite unramified tower $(k_i)$ of totally real fields with $rd_{k_i}\le c_0$ provided by Theorem \ref{thm:GS}. As the degrees $d_i\to\infty$, we can assume that $d_i\ge t$ for all $i$. Let $k = k_i$ be one of the fields.
Let $\theta_1,\dots,\theta_t \in k$ be Pisot numbers chosen as in part (b) of the lemma and let
$\alpha = (1-\theta_1)\ldots(1-\theta_t)$.
Then the field $l = k[\sqrt{\alpha}]$ has precisely $t$ complex places and we have the following bound for its discriminant:
\begin{align*}
& \D_l \le \D_k^2 2^{2d} |N(\alpha)| \le \D_k^2 2^{2d} \D_k^{t\delta}; \\
& rd_l \le 2\D_k^{\frac{2+t\delta}{2d}} \le 2c_0^{\frac{2+t\delta}{2}}
\end{align*}
(here the first inequality follows from \cite[Prop. 8, p. 62]{La}, \cite[Prop. 14, p. 66]{La} and some elementary properties of the norm).
Repeating this procedure for all $k_i$ we obtain an asymptotically bounded sequence of fields with the required properties.
\end{proof}

From Minkowski's theorem it follows that there exists a positive lower bound for the constants $c_0$ of asymptotically bounded sequences of fields.
Although we do not require it in this paper, it would be interesting to know more about this bound and its dependence on the number of complex places of the fields in the sequences.

\section{Arithmetic subgroups and their covolumes}\label{section:arithmetic}

\subsection{}\label{sec:ar}
Let $H$ be a semisimple connected linear Lie group without compact factors. It
is known that if $H$ contains irreducible lattices then all of its almost
simple factors are of the same type. Such groups $H$ are called {\it isotypic} or {\it
typewise homogeneous} (see~\cite[Chapt.~9.4]{Ma}). So from now on we shall assume
that $H$ is isotypic. Moreover, without loss of generality we can further assume
that the center of $H$ is trivial. This implies that $H$ is isomorphic to ${\rm
Ad\:} H$, where ${\rm Ad}$ denotes as usual the adjoint representation. The
group ${\rm Ad\:} H$ is the connected component of identity of the $\R$-points
of a semisimple algebraic $\R$-group. There exist, therefore,
absolutely simple $\R$-groups $\G_i$, all of the same type, such that $H = (\prod_{i=
1}^{a}\G_i(\R) \times \G_{a+1}(\C)^{b})^o.$ A classical theorem of Borel
\cite{Bo1} (see also \cite{BH}) asserts that such $H$ does contain irreducible
lattices.

Let now $\G$ be an algebraic group defined over a number field $k$ which admits
an epimorphism $\phi:\G(k\otimes_\Q\R)^o \to H$ whose kernel is compact. In
this case, $\phi(\G(\cO))$ is an irreducible lattice in $H$.
Such lattices and the subgroups of $H$ which are
commensurable with them are called {\it arithmetic}. It can be shown that to
define all arithmetic subgroups of $H$ it is sufficient to consider only
simply connected, absolutely almost simple $k$-groups $\G$ which have the same
(absolute) type as the almost simple factors of $H$ and are defined over the
fields with at most $b$ complex and at least $a$ real places. In this case, as $\G$ 
is a simply connected $k$-group, $\G(k\otimes_\Q\R)$ is connected. We shall call 
such groups $\G$ and corresponding fields $k$ {\it admissible}.

The local-global principle provides a standard way to construct arithmetic
subgroups which will be particularly useful for us. Let $\Pa = (\Pa_v)_{v\in V_f}$
be a collection of parahoric subgroups $\Pa_v\subset\G(k_v)$
of a simply connected $k$-group $\G$. The family $\Pa$ is called {\it coherent} if
$\prod_{v\in V_\infty}\G(k_v)\cdot\prod_{v\in V_f} \Pa_v$ is an open subgroup of
the ad\`ele group $\G(\A_k)$.
Now let
\begin{equation*}
\Lambda = \Lambda(\Pa) = \G(k)\cap\prod_{v\in V_f} \Pa_v,
\end{equation*}
where $\Pa$ is a coherent collection. Following \cite{Pr}, we shall call $\Lambda$ the
{\em principal arithmetic subgroup} associated to $\Pa$. We shall also call $\Lambda' = \phi(\Lambda)$
a principal arithmetic subgroup of $H$.

\subsection{}\label{sec:ar_vol}
The Lie group $H$ carries a Haar measure $\mu$ which is uniquely defined up to a constant
factor. The choice of a particular normalization of $\mu$ is not essential for
our considerations. From now on we shall fix a Haar measure on
$\G(k\otimes_\Q\R)$ for some admissible $\G/k$ following \cite[Secs.~1.4,
3.6]{Pr}, this also defines a normalized Haar measure on $H$ which does
not depend on the choice of $\G$. We can compute the covolumes of principal
arithmetic subgroups with respect to $\mu$ using Prasad's volume formula. By
\cite[Theorem~3.7]{Pr}, we have:
\begin{equation*}
\mu(H/\Lambda')  = \D_k^{\dm(\G)/2}(\D_l/\D_k^{[l:k]})^{\frac12s}
\left(\prod_{i=1}^{r}\frac{m_i!}{(2\pi)^{m_i+1}}\right)^{[k:\Q]}
\tau_k(\G)\:\E(\Pa),
\end{equation*}
where
\begin{itemize}
\item[(i)] $\dm(\G)$, $r$ and $m_i$ denote the dimension, rank and Lie exponents of $\G$;
\smallskip\item[(ii)] $l$ is a Galois extension of $k$ defined as in \cite[0.2]{Pr}
(if $\G$ is not a $k$-form of type $^6\Dn_4$, then $l$ is the split field of
the quasi-split inner $k$-form of $\G$, and if $\G$ is of type $^6\Dn_4$, then
$l$ is a fixed cubic extension of $k$ contained in the corresponding split
field; in all the cases $[l:k]\le 3$);
\smallskip\item[(iii)] $s = s(\G)$  is an integer defined in  \cite[0.4]{Pr},
in particular, $s=0$  if  $\G$  is an inner form of a split group and $s\ge 5$
if $\G$ is an outer form;
\smallskip\item[(iv)] $\tau_k(\G)$ is the Tamagawa number of $\G$ over $k$ (since $\G$ is simply connected
and $k$ is a number field, $\tau_k(\G) = 1$); and
\smallskip\item[(v)]
$\E(\Pa) = \prod_{v\in V_f} e_v$ is an Euler product of the local factors $e_v =
e(\Pa_v)$.
\end{itemize}

The local factors $e_v$ can be effectively computed using the Bruhat-Tits
theory. In order to justify this claim we will need a few more definitions.

\subsection{} \label{sec:BT}
Let $k_v$ be a nonarchimedean local field of characteristic zero (a finite extension of
the $p$-adic field $\Q_p$), and let $\G$ be an absolutely almost simple, simply connected
$k_v$-group. The Bruhat-Tits theory~\cite{BT} associates to $\G/k_v$ a simplicial complex
$\B = \B(\G/k_v)$ on which $\G(k_v)$ acts by simplicial automorphisms. The complex $\B$
is called the {\it affine building} of $\G/k_v$. A {\it parahoric subgroup} $\Pa$
of $\G(k_v)$ is defined as a stabilizer of a simplex of $\B$. Every parahoric
subgroup is compact and open in $\G(k_v)$ in the $p$-adic topology. Maximal
parahoric subgroups are the maximal compact subgroups of $\G(k_v)$; they are
characterized by the property of being the stabilizers of the vertices of $\B$. A
maximal parahoric subgroup is called {\it special} if it fixes a {\it special
vertex} of $\B$. A vertex $x\in\B$ is special if the affine Weyl group $W$ of
$\G(k_v)$ is a semidirect product of the translation subgroup by the isotropy
group $W_x$ of $x$ in $W$. In this case, $W_x$ is canonically isomorphic to the
(finite) Weyl group of the $k_v$-root system of $\G$. If $\G$ is quasi-split over $k_v$
and splits over an unramified extension of $k_v$, then $\G(k_v)$ contains also
{\it hyperspecial} parahoric subgroups (see \cite[1.10]{Tits}); these subgroups are
parahoric subgroups of $\G(k_v)$ of the maximal volume (\cite[3.8.2]{Tits}).

Every special (or hyperspecial) parahoric subgroup $\Pa_v$
has a normal pro-$p$ subgroup the quotient by which is a quasi-simple group (i.e.
it is simple modulo the center), and hence it also has a maximal prosolvable
normal subgroup with a finite simple (non-abelian) quotient. Such a maximal
normal subgroup is unique.

Following \cite{Pr}, we associate to a parahoric subgroup $\Pa_v\subset\G(k_v)$ two
reductive groups $\cMv$ and $\Mv$ over the residue field $\F_v$ of $k_v$:
Using the Bruhat-Tits theory one can define a smooth affine
group scheme $\G_v$ over the ring of integers $\cO_v$ of $k_v$,
whose generic fiber ($= \G_v\times_{\cO_v} k_v$)
is isomorphic to $\G(k_v)$ and whose group of integral points is
isomorphic to $\Pa_v$. Then $\Mv$ denotes a maximal connected reductive
$\F_v$-subgroup of $\G_v \times_{\cO_v} \F_v$. The group $\cMv$ is
defined in a similar way for the quasi-split inner form $\qsG$ of $\G(k_v)$ and a
specially chosen parahoric subgroup of $\qsG$. We refer to \cite[2.2]{Pr} for
the details and finally write down the expression for the local factor $e_v$
in the volume formula:
\begin{equation*}
e_v = e(\Pa_v) = \frac{\#\F_v^{\:(\dm(\Mv) + \dm(\cMv))/2}}
{\#\Mv(\F_v)}.
\end{equation*}

Assume now that $\G$ is quasi-split over $k_v$ and $\Pa_v$ is a special parahoric
subgroup, which is, moreover, assumed to be hyperspecial
if $\G$ splits over an unramified extension of $k_v$. In this case, $\cMv$
is isomorphic to $\Mv$ and $\Mv(\F_v)$ is a finite simple group of the
same type as $\G$. So the computation of $e(\Pa_v)$ becomes easy (see
\cite[Rem.~3.11]{Pr} and \S\ref{sec:thm1 lower} below). We recall that these conditions on
$\G$ and $\Pa_v$ are indeed satisfied for almost all nonarchimedean places
of $k$: $\G$ is quasi-split over almost every $k_v$ and
$\prod_{v\in V_\infty}\G(k_v)\cdot\prod_{v\in V_f} \Pa_v$ being open in
$\G(\A_k)$ implies that $\Pa_v$ is hyperspecial for almost every $v$. Thus
generically the computation of the local factors in the volume formula is
pretty straightforward.

\subsection{} \label{sec:max_ar}
Let $\Gamma$ be a maximal arithmetic lattice in $H$. It is known that $\Gamma$ can be obtained as a normalizer in $H$ of the image $\Lambda'$ of some principal arithmetic subgroup $\Lambda$ of $\G(k)$ (see \cite[Prop.~1.4(iv)]{BP}). Moreover, such $\Lambda$'s are principal arithmetic subgroups of {\em maximal type} in a sense of Rohlfs (see \cite{Rohlfs} and also \cite{CR} for precise definitions). In order to prove the main theorem we will need certain control over the structure of $\Lambda$ and the index $[\Gamma:\Lambda']$ in terms of the covolume of $\Gamma$. For this purpose we recall two results which follow from \cite{B}.

\medskip

Let $\Gamma = N_H(\Lambda')$ ($\Lambda' = \phi(\Lambda)$, $\Lambda = \G(k)\cap\prod_{v\in V_f} \Pa_v$) be a maximal arithmetic lattice of covolume less than $x$, with $x$ large enough.

\begin{prop} \label{cor1:B}
Let $T$ be the smallest set of nonarchimedean places of $k$ such that for every $v\in V_f\smallsetminus T$, $\G$ is quasi-split over $k_v$, splits over an unramified extension of $k_v$, and $\Pa_v$ is hyperspecial. Then there exists a constant $C_1 = C_1(H)$ such that $\prod_{v\in T} q_v \le x^{C_1}$.
\end{prop}
\begin{proof} This result follows from \cite[Secs.~4.1, 4.3, 4.4]{B} but is not stated there explicitly. We recall the main steps of the proof.

Let $T_1$ be the subset of the nonarchimedean places of $k$ such that $\G$ is not quasi-split over $k_v$ for $v\in T_1$, let $R\subset V_f$ be the set of places for which $\G$ is quasi-split but is not split over an unramified extension of $k_v$, and let $T_2\subset V_f\smallsetminus(T_1\cup R)$ be the set of places for which $\Pa_v$ is not hyperspecial. Then $T = T_1 \cup R \cup T_2$ is a finite subset of $V_f$. Moreover,
\begin{align*}
\mu(H/\Gamma) &\ge c_1\prod_{v\in T_1} q_v^{\delta_1}, \text{ by \cite[4.3]{B};}\\
\mu(H/\Gamma) &\ge c_2\left(\D_l/\D_k^{[l:k]}\right)^{\delta_2} \ge c_2\prod_{v\in R} q_v^{\delta_2}, \text{ by \cite[4.1]{B},
see also \cite[4.3]{B};}\\ 
\mu(H/\Gamma) &\ge c_3\prod_{v\in T_2} q_v^{\delta_3}, \text{ by \cite[4.4]{B},}
\end{align*}
where $c_1$, $c_2$, $c_3 >0$ are some absolute constants and $\delta_1$, $\delta_2$, $\delta_3 > 0$ are constants which depend only on the Lie type of $H$.

Altogether, these inequalities imply that there exist $c>0$ and $\delta = \delta(H) > 0$ such that $x \ge \mu(H/\Gamma) \ge c\prod_{v\in T} q_v^{\delta}$, and the proposition follows.
\end{proof}

\begin{prop} \label{cor:B} \cite[Cor.~6.1]{B}
There exists a constant $C_2 = C_2(H)$ such that for $Q = \Gamma/\Lambda'$ we have $|Q| \le x^{C_2}$.
\end{prop}

\subsection{} \label{sec:level vs index}
For future use let us give a variant of the "level versus index" lemma where the level is controlled by the covolume of the lattice. To put it in a perspective, recall the classical lemma asserting that in $\Delta = \SL_2(\Z)$, every congruence subgroup of index $n$ contains $\Delta(m) = \Ker(\SL_2(\Z)\to\SL_2(\Z/m\Z))$ for some $m \le n$, i.e. the level $m$ is at most the index $n$. This was generalized in \cite{Lu} to the congruence subgroups of an arbitrary arithmetic group $\Delta$ by paying a price for $m$; i.e. it was shown that $m\le Cn$ for some constant $C$ which depends on the arithmetic group $\Delta$. Here we want to bound $C$ in terms of the covolume.

Let us first introduce some notations. As before, let $\Lambda = \G(k)\cap\prod_{v\in V_f}\Pa_v$
where $k$ is a number field with the ring of integers $\cO$, $\G$ is a $k$-form of $H$ and $\Pa_v$ is a parahoric subgroup of $\G(k_v)$, and  let $\G_v$ be an $\cO_v$-scheme  with the
generic fiber isomorphic to $\G(k_v)$ such that $\G_v(\cO_v) = \Pa_v$. This
induces a congruence subgroup structure on $\Pa_v$ defined as follows:
\begin{equation*}
\Pa_v(r) = \Ker (\G_v(\cO_v) \to \G_v(\cO_v/\pi_v^r\cO_v)),
\end{equation*}
where $\pi_v$ is a uniformizer of $\cO_v$. These congruence subgroups induce a
congruence structure on $\Lambda$, $\Lambda(\pi_v^r) = \Pa_v(r)\cap\Lambda$. More
generally, for every ideal $I$ of $\cO$ look at its closure $\bar I$ in
$\hat\cO = \prod_v \cO_v$. Then $\bar I$ is equal to $\prod_{i=1}^{l}
\pi_{v_i}^{e_i}\hat\cO$ for some $Y = \{v_1, \ldots, v_l\}\subset V_f$ and
$e_1,\ldots ,e_l \in \N$. We then define the $I$-congruence subgroup of $\Lambda$,
\begin{equation*}
\Lambda(I) = \Lambda \cap (\prod_{i=1}^{l} \Pa_{v_i}(e_i)\cdot\prod_{v\not\in Y}\Pa_v).
\end{equation*}
In particular, for every $m\in\N$, the $m$-congruence subgroup $\Lambda(m) = \Lambda(m\cO)$ is
defined. Any subgroup of $\Lambda$ which contains $\Lambda(I)$ for some non-zero ideal $I$ is
called a {\em congruence subgroup}.

Let now $\Lambda$ be a principal arithmetic subgroup of a maximal type in $\G(k)$ and let $\Lambda'$ be its image in $H$. Assume also that
$\mu(H/\Lambda') \le x$, where $x\gg 0$.
\begin{lemma} \label{level vs index}
If $\Lambda_1$ is a congruence subgroup of $\Lambda$ of index $n$, then $\Lambda_1 \supseteq \Lambda(m\cO)$ where $m\in\N$ with $m\le x^C n$ and $C$ is a constant which depends only on $H$.
\end{lemma}
\begin{proof}
A similar result is proved in \cite[Prop. 6.1.2]{LS} but the proposition there provides only $m\le C_0n$ for some constant $C_0$ depending on $\Lambda$. In fact, the proof of the proposition gives $C_0 = 1$ if certain conditions (i)--(iv) are satisfied for all primes. The role of $C_0$ is to compensate for the bad primes. Now, if $\Lambda$ is a principal arithmetic subgroup of a maximal type as described above, then the conditions (i)--(iv) are satisfied for all the primes $v\in V_f\smallsetminus T$, where $T$ is the set from Proposition \ref{cor1:B}. We need to compensate for the primes $v\in T$. For each one of them, we can start the induction argument in the proof of Proposition 6.1.2 \cite{LS} from the first congruence subgroup so, by Proposition \ref{cor1:B}, we can replace $C_0$ by $x^C$ for some constant $C$ depending only on $H$.
\end{proof}

\begin{rem}
Note that the index of $m\cO$ in $\cO$ (and hence also of $\Lambda(m\cO)$ in $\Lambda$) is not necessarily polynomial in $m$, but rather it is bounded by $m^d$ where $d$ is the degree of the defining field $k$ of the arithmetic subgroup $\Lambda$. As $d$ is bounded by $O(\log x)$, the index of $\Lambda(m\cO)$ in $\Lambda$ is bounded by $(xn)^{c\log x}$.
A better result is probably true: $\Lambda_1 \supseteq \Lambda(I)$ for some $I\lhd\cO$ such that $[\Lambda:\Lambda(I)]\le (xn)^c$ with a constant $c$ depending only on $H$. This indeed follows from Lemma \ref{level vs index} if the degree of the field $k$ is bounded.
\end{rem}

\section{Counting covers versus counting manifolds} \label{section:covers vs manifolds}

The results of this paper rely heavily on ``subgroup growth'' (\cite{LS})
but there is a crucial difference: If $M$
is a finite volume manifold covered by a symmetric space $X = H/K$ ($H$ is a
semisimple Lie group and $K$ is a maximal compact subgroup of $H$) with $\Gamma
= \pi_1(M)$, then there is a one-to-one correspondence between the $n$-sheeted
{\bf covers} of $M$ and the $\Gamma$-conjugacy classes of index $n$ subgroups
of $\Gamma$. Thus, if $a_n(\Gamma)$ denotes the number of subgroups of $\Gamma$
of index $n$ and $b_n(M)$~--- the number of $n$-sheeted covers of $M$, then
$$b_n(M) \le a_n(\Gamma) \le n b_n(M).$$
Thus, counting subgroups and counting covers are essentially the same, up to a
linear factor. On the other hand, in this paper we count {\bf manifolds}, so
two covers of $M$ are identified if they are isomorphic as manifolds even if
they are not isomorphic as covers. In group theoretic terms it means that we
are counting ${\rm Iso} (X)$-conjugacy classes of lattices, where ${\rm Iso} (X)$
is the group of isometries of $X$. Now, $H$ is of finite index in ${\rm Iso} (X)$
and so, up to a constant factor, we are counting $H$-conjugacy classes
of lattices in $H$.
\medskip

Ideally, what we would like to have is:
\begin{conj}\label{conj_s3}
There exists a constant $c = c(H)$, such that if $\Gamma$ is a lattice in $H$
and $\Gamma_1$ is a subgroup of $\Gamma$ of covolume at most $x$ in $H$, then
the number of subgroups of $\Gamma$ which are $H$-conjugate to $\Gamma_1$ is
bounded by $x^c$ if $x$ is large enough.
\end{conj}

We do not know if this conjecture is true or just a wishful thinking. In this
section we shall establish a weaker version which will suffice for our
applications.

\medskip

Observe first that if $\Gamma_1$ and $\Gamma_2$ are index $n$ subgroups of a
lattice $\Gamma$ in $H$, then $\Gamma_1\backslash H/K$ is isometric to
$\Gamma_2\backslash H/K$ if and only if there exists $h\in {\rm Iso} (X)$ which
conjugates $\Gamma_1$ to $\Gamma_2$, i.e. $\Gamma_1$ and $\Gamma_2$ are
conjugate in ${\rm Iso} (X)$. For counting purposes (up to a constant factor)
we can assume $h\in H$. Such an $h$ conjugating $\Gamma_1$ to $\Gamma_2$ is an
element of the commensurability group $\Comm (\Gamma) = \{ h\in H \mid [\Gamma
: \Gamma \cap h^{-1}\Gamma h] < \infty \}$. Recall that if $\Gamma$ is
non-arithmetic irreducible lattice in $H$, then $[\Comm (\Gamma):\Gamma] <
\infty$ by a well known result of Margulis (\cite[Theorem~1, p.~2]{Ma}). This
implies that counting covers of a non-arithmetic manifold $M$ is, up to a
constant factor (depending on $M$, though), the same as counting manifolds
covering $M$. This is the reason why in~\cite{BGLM} the lower bound on the
number of hyperbolic manifolds was presented using covers of non-arithmetic
manifolds. A similar remark applies in a different context to~\cite{BL1}. But,
in this paper, when we deal with the higher-rank $H$, all lattices are
arithmetic and so we must consider the delicate issue of the difference
between isomorphism classes of covers and isomorphisms of manifolds.

\medskip

Let $\G$ be an absolutely simple, simply connected algebraic group defined over
a number field $k$ and let us fix a $k$-embedding $\G \subset \GL_s$ for some $s$.
Let $\mathrm{Z}$ denote the center of $\G$ and $\pi: \G\to\adG := \G/\mathrm{Z}$ be the
natural projection  defined over $k$. If $\Gamma$ is commensurable to $\G(\cO) = \G(k)
\cap \GL_s(\cO)$ ($\cO$ is the ring of integers of $k$), then
$\pi(\Gamma)\subset\adG(k)$, ${\rm Comm}_\G\:(\Gamma) = {\rm Comm}_\G\:(\G(\cO))$ and $\pi({\rm
Comm}_\G\:(\Gamma))$ is also in $\adG(k)$ (see e.g. \cite[Lemma VII.6.2]{Ma}).

For every $v\in V_f$, let $\Pa_v$ be a maximal parahoric subgroup of $\G(k_v)$
such that $\Pa = (\Pa_v)_{v\in V_f}$ is a coherent collection. By the Bruhat-Tits
theory (see \S\ref{sec:BT}), for every $v$ there exists a smooth affine
group scheme $\G_v$ defined over $\cO_v$, the ring of integers of $k_v$, such
that $\G_v(\cO_v) = \Pa_v$ and $\G_v(k_v)$ is $k_v$-isomorphic to $\G(k_v)$. Let
$\Kv$ be the normal pro-$p$ subgroup of $\Pa_v$, $\Kv = \Ker(\G_v(\cO_v) \to
\G_v(\F_{q_v}))$ where $\F_{q_v} = \cO_v/m_v$ is the residue field of $\cO_v$
w.r.t. the maximal ideal $m_v$.

Recall that when $\G(k_v) = \G_v(k_v)$ acts on the Bruhat-Tits building $\B_v$
associated with it, $\Pa_v$ is the stabilizer of some vertex $w_v\in\B_v$ and
$\Kv$ is the set of elements of $\G(k_v)$ which fixes pointwise the link of
$w_v$. This link is isomorphic to the projective building of the finite group
$\G_v(\F_{q_v})$, in particular, this implies that the number of vertices of
the link is at most $\#\G_v(\F_{q_v}) \le q_v^{\dm(\G)}$.

Let $\Lambda$ be the principal arithmetic subgroup of $\G(k)$ associated with
$\Pa = (\Pa_v)_{v\in V_f}$ as in \S\ref{sec:ar}, i.e. $\Lambda = \G(k)\cap\prod_{v\in
V_f}\Pa_v$. Let $I\subset V_f$ be a fixed finite subset of nonarchimedean
places of $k$. It defines an ideal of $\cO$ which we denote by the same letter.
The group $\Lambda$ is embedded diagonally in $\prod_{v\in I} \G(k_v)$.
Let $\Lambda(I) = \Lambda \cap \prod_{v\in I} \Kv$, the $I$-congruence
subgroup of $\Lambda$. It is a finite index normal subgroup of $\Lambda$ and
the index is bounded by $\prod_{v\in I}q_v^{\dm(\G)}$.

Denote by $\oG$ the image $\pi(\Lambda)$ of $\Lambda$ in $\adG(k)$. The group $\adG(k)$ acts on $\G(k)$ by the adjoint action. Let
$$N(\GI, \Lambda) = \{ g\in\adG(k) \mid g(\GI) \subseteq \Lambda \}.$$
This is not a subgroup but rather a union of finitely many cosets of $\oG$ including $\oG$ itself. We call the number of these cosets the {\it index of $\oG$ in $N(\GI, \Lambda)$} and denote it by $[N(\GI, \Lambda):\oG]$.

\begin{prop}\label{prop_s3}
Let $\Lambda$ be a principal arithmetic subgroup associated with $\Pa = (\Pa_v)_{v\in V_f}$, such that $\Pa_v$ is a maximal parahoric
subgroup for every $v$. If $\mu(H/\Lambda') \le x$, then for every ideal $I$ as above,
$$[N(\GI, \Lambda):\oG] \le x^C\left( \prod_{v\in I}q_v \right)^{\dm(\G)},$$
where $C = C(H)$ is a constant.
\end{prop}

\begin{proof}
Denote
\begin{gather*}
\cP = \prod_{v\in V_f}\Pa_v \subset \prod_{v\in V_f}{}^{\!\!\!'}\G(k_v)\quad \text{and} \quad \cK = \prod_{v\in I} \Kv \times \prod_{v\in V_f\smallsetminus I}\Pa_v.
\end{gather*}
Then, by the strong approximation theorem \cite[Theorem 7.12, p. 427]{PR}, $\Lambda$ (resp. $\GI$) is dense in $\cP$ (resp. $\cK$) and $\cP\cap\G(k) = \Lambda$ (resp. $\cK\cap\G(k) = \GI$), when $\G(k)$ is embedded diagonally in $\prod_{v\in V_f}'\G(k_v)$.

Let now
$$N(\cK,\cP) = \{ g\in\adG(\A_f) \mid g(\cK) \subseteq \cP \}.$$
It is easy to see that $N(\cK,\cP) \supseteq N(\GI, \Lambda) \supseteq \oG$. Indeed, if $g\in N(\GI, \Lambda)$, then it is in $N(\cK,\cP)$ by the density of $\Lambda$ (resp. $\GI$) in $\cP$ (resp. $\cK$) and the continuity of the action. The second inclusion is obvious. This implies
$$[N(\GI, \Lambda):\oG] \le [N(\cK, \cP):\ocP] \cdot [N(\Lambda, \Lambda):\oG],$$
where $\ocP  = N(\cP, \cP) = \prod_{v\in V_f}\overline{\Pa}_v$, $\overline{\Pa}_v$ is the stabilizer of $\Pa_v$ in $\adG(k_v)$, and $\ocP\cap\adG(k) = N(\Lambda,\Lambda)$. Now, by Proposition \ref{cor:B}, $[N(\Lambda, \Lambda):\oG] \le x^C$.

If $v\in V_f\smallsetminus I$, then the projections of $\cK$ and $\cP$ to $\G(k_v)$ are both $\Pa_v$, so if $g \in N(\cP,\cK)$, its $v$-component is in the stabilizer  $\overline{\Pa}_v$ of $\Pa_v$.
For $v\in I $, let us denote $N(\Kv, \Pa_v) = \{ g\in\adG(k_v) \mid g(\Kv) \subseteq \Pa_v \}$  and let $[N(\Kv, \Pa_v):\overline{\Pa}_v]$ denotes the number of $\overline{\Pa}_v$-cosets in $N(\Kv, \Pa_v)$. Clearly, $N(\cK,\cP)$ is contained in $\prod_{v\in I}N(\Kv, \Pa_v)\times\prod_{v\in V_f\smallsetminus I}N(\Pa_v, \Pa_v)$, which implies
$$
[N(\cK, \cP):\ocP] =
\prod_{v\in I} [N(\Kv, \Pa_v):\overline{\Pa}_v] \cdot \prod_{v\in V_f\smallsetminus I} [N(\Pa_v, \Pa_v):\overline{\Pa}_v] = \prod_{v\in I} [N(\Kv, \Pa_v):\overline{\Pa}_v].
$$
We shall show that $[N(\Kv, \Pa_v):\overline{\Pa}_v] \le q_v^{\dm(\G)}$ which will finish the proof.

The subgroup $\Pa_v$, being a maximal parahoric subgroup of $\G(k_v)$, is the
stabilizer of a vertex $w_v$ of $\B_v$ and $\Kv$ is the subgroup of
$\G(k_v)$ which fixes pointwise all the vertices $w\in\B_v$ of distance at most
$1$ from $w_v$, and the fixed point set of $\Kv$ is exactly this set.
Thus if $g\in N(\Kv, \Pa_v)$, then the fixed point set of
$g(\Kv)$ includes $w_v$, which is equivalent to $g(w_v)$ being fixed by
$\Kv$, i.e., $g(w_v)$ is of distance $\le 1$ from $w_v$. As it was pointed out
above, the link of a vertex of the Bruhat-Tits building of $\G(k_v)$ has order at
most $q_v^{\dm(\G)}$. The number of cosets of $\overline{\Pa}_v$ in $N(\Kv, \Pa_v)$ is,
therefore, also bounded by $q_v^{\dm(\G)}$ and the proposition is now proven.
\end{proof}

\begin{cor}\label{cor_s3}
If $\Lambda_1$ is a subgroup of index $n$ in $\Lambda$ containing $\GI$, then the
number of subgroups of $\Lambda$ which are conjugate to $\Lambda_1$ within $\adG(k)$ is
bounded by $n x^C \left(\prod_{v\in I}q_v \right)^{\dm(\G)}.$
\end{cor}

\begin{proof}
The number of $\Lambda$-conjugates of $\Lambda_1$ is at most $n$. Now, if $g\in\adG(k)$ and $g(\Lambda_1) = \Lambda_2\subseteq\Lambda$, then $g(\GI) \subseteq \Lambda$, and so $g\in N(\GI, \Lambda)$. The latter contains at most $x^C \left(\prod_{v\in I}q_v \right)^{\dm(\G)}$ cosets of $\oG$ by the proposition, therefore the total number of possibilities for $\Lambda_2$ is bounded by $n x^C \left(\prod_{v\in I}q_v \right)^{\dm(\G)}.$
\end{proof}

\section{Proof of the lower bound}\label{sec:thm1 lower}

Our strategy will be the following: By using an asymptotically
bounded sequence of fields $k_i$ of degree
$d_i$ over $\Q$, we shall construct principal arithmetic subgroups $\Lambda_i$
in $H$ of covolume bounded by $c_1^{d_i}$ for some constant $c_1$. We then
present in each $\Lambda_i$, $c_2^{d_i^2}$ subgroups of index at most
$c_3^{d_i}$ (where $c_1$, $c_2$, $c_3$ are constants $>1$). We further
show that ``generically'' these subgroups are not conjugate to each other. We
therefore can deduce that asymptotically $H$ has at least $c_2^{d_i^2}$
non-conjugate lattices of covolume at most $(c_1c_3)^{d_i}$. This will prove
the lower bound in Theorem \ref{theorem} with $a = \log c_2/ (\log c_1c_3)^2$.

\medskip

If $H$ is a real simple Lie group, let $(k_i)$ be a totally real infinite class field tower as in Theorem \ref{thm:GS}, and if $H$ is complex let $(k_i)$ be an asymptotically bounded sequence provided by Corollary \ref{cor:Pisot} with $t = r_2(k_i) = 1$. In both cases $d_i = d_{k_i} \to \infty$ and $rd_i = \D_{k_i}^{1/d_i} \le c_0$, for an absolute constant $c_0$.

%
%

Let $k = k_i$ be one of the fields and $d = d_k$. In order to construct arithmetic lattices in $H$ which are defined over $k$ and have certain properties we appeal to results of \cite{BH} and \cite{PR06}.

Let $\tilde{H}$ be the simply connected cover of $H$ and let $\tilde{H}_{cpt}$ be its compact real form. Recall that the real groups of types $\Bn_n$, $\Cn_n$, $\En_7$, $\En_8$, $\Fn$ and $\Gn$ are inner, while types $\An_n$, $\Dn_n$ and $\En_6$ admit both inner and outer real forms. Moreover, compact groups of types $\An_n$ ($n>1$), $\Dn_{2n+1}$ and $\En_6$ are outer (cf. \cite{Tits9}). We define an extension $l$ of the field $k$ as follows:
\begin{itemize}
\item[(i)] If $\tilde{H}$ is either complex or it is real and inner and if $\tilde{H}_{cpt}$ is inner, let $l = k$;
\item[(ii)] If $\tilde{H}$ is either complex or it is real and outer and if $\tilde{H}_{cpt}$ is outer, let $l$ be a quadratic extension of $k$ such that the real places of $k$ do not split in $l$;
\item[(iii)] If $\tilde{H}$ is real outer and $\tilde{H}_{cpt}$ is inner, let $l$ be a quadratic extension of $k$ such that $v_1\in V_\infty(k)$ does not split in $l$ while all the rest real $v$ split;
\item[(iv)] If $\tilde{H}$ is real inner and $\tilde{H}_{cpt}$ is outer, let $l$ be a quadratic extension of $k$ such that $v_1\in V_\infty(k)$ splits in $l$ while all the rest real $v$ do not split in $l$.
\end{itemize}
(We say that a real place $v$ of $k$ splits in a quadratic extension $l$ if there exist two extensions of $v$ to $l$.)

Note that we can always choose $l$ so that $\D_{l/k} \le {c'_0}^{d}$ with some absolute constant $c'_0$: In case (i) it is clear. In case (ii) we can take $l = k[i]$ for which $c_0' = 4$ (as only primes of $k$ which lie over $2$ may possibly ramify in $k[i]$). In case (iii), let  $l = k[\sqrt{1-\theta}]$, and in case (iv), $l = k[\sqrt{\theta-1}]$, where $\theta$ is a Pisot number in $k$ provided by Lemma \ref{lemma:Pisot}(a). To show that in the last two cases $\D_{l/k} \le {c'_0}^{d}$ we can apply the same argument as in Corollary \ref{cor:Pisot}.

\medskip

Let $p_0$ be a fixed rational prime and let $v_0$ be a fixed place of $k$ above $p_0$.

\begin{prop} There exists an absolutely simple simply connected $k$-group $\G$ such that
\begin{itemize}
\item[(1)] $\G(k\otimes_\Q\R)$ admits an epimorphism to $H$ whose kernel is compact (i.e. $\G$ is admissible in the sense of \ref{sec:ar});
\item[(2)] $\G$ is quasi-split over $k_v$ for every $v\in V_f\smallsetminus \{ v_0 \}$;
\item[(3)] The quasi-split inner form of $\G$ splits over $l$.
\end{itemize}
\end{prop}

\begin{proof}
Let $\G_0$ be an absolutely simple, simply connected, quasi-split $k$-group of the same absolute type as $H$ which splits over $l$ and does not split over $k$ if $k\neq l$. Similarly to \cite[Props. 4, 5]{PR06} it follows from \cite[Theorem 1(i)]{PR06} that there exists an inner twist $\G$ of $\G_0$ over $k$ which satisfies (1) and (2). Property (3) is satisfied automatically by the definition of $\G_0$, which is the quasi-split inner form of $\G$.
\end{proof}

Let $\Pa = (\Pa_v)_{v\in V_f}$ be a coherent collection of parahoric subgroups
of $\G$ such that for every $v\neq v_0$, $\Pa_v$ is special and it is hyperspecial whenever $l$ is unramified over $k$ at $v$.
Let $\Lambda = \G(k)\cap\prod_{v\in V_f}\Pa_v$ be the corresponding principal
arithmetic subgroup of $\G(k)$. By the definition of $\G$, the projection $\Lambda' = \phi(\Lambda)$ (induced by
$\phi : \G(k\otimes_\Q\R) \to H$) is an arithmetic lattice in $H$. We shall now use
Prasad's formula (see \S\ref{sec:ar_vol}) to compute its covolume:
\begin{equation*}
\mu(H/\Lambda')  = \D_k^{\dm(\G)/2}(\D_{l}/\D_k^{[l:k]})^{\frac12s}
\left(\prod_{i=1}^{r}\frac{m_i!}{(2\pi)^{m_i+1}}\right)^{[k:\Q]}
\tau_k(\G)\:\E(\Pa).
\end{equation*}

By the construction, the field $l$ in the volume formula is the extension of $k$
defined above, so we have
\begin{equation*}
\D_k \le c_0^{d},\ \  \D_{l}/\D_k^{[l:k]} = \D_{l/k} \le {c'_0}^d.
\end{equation*}
Since $\G$ is a simply connected group over a number field $k$, the Tamagawa
number $\tau_k(\G) = 1$.
It remains to analyze the Euler product
$$\E(\Pa) = \prod_{v\in V_f}\frac{q_v^{(\dm(\Mv) + \dm(\cMv))/2}}{\#\Mv (\F_v)}.$$

For $v\neq v_0$, $\G(k_v)$ is quasi-split and $\Pa_v$ is special, so  $\Mv$ is
isomorphic to $\cMv$ over $\F_v$ and $\Mv(\F_v)$ is a finite simple group of the
same type as $\G$. Indeed, since $\Pa_v$ is a maximal parahoric subgroup, the
radical of $\Mv$ is trivial, so $\Mv(\F_v)$ is a finite semisimple group whose
diagram can be obtained by deleting the vertex corresponding to $\Pa_v$ and all
the adjacent edges from the extended Dynkin diagram of $\G(k_v)$. It remains to
recall the definition of the special parahoric subgroups to see that
$\Mv(\F_v)$ is a simple group of the same type as $\G$. So the order of
$\Mv(\F_v)$ is known (see e.g.~\cite{Ono}):
\begin{equation*}
\#\Mv (\F_v) = q_v^{\dm(\Mv)}\prod_{i=1}^{r}(1 \pm q_v^{-(m_i+1)}),
\end{equation*}
(except for the groups of type $\Dn_4$ whose splitting field is of degree $3$ over $\F_v$, but these groups do not arise in our setting). The sign $\pm$ in the formula depends on the splitting type of $\Mv(\F_v)$.

In all the cases we obtain (for $v\neq v_0$):
\begin{equation*}
\#\Mv (\F_v) \ge q_v^{\dm(\Mv)}\prod_{i=1}^{r}(1 - q_v^{-(m_i+1)}).
\end{equation*}
Also, as $\Mv$ is isomorphic to $\cMv$ over $\F_v$, we have $\frac{\dm(\Mv) + \dm(\cMv)}{2} = \dm(\Mv)$.

We now can bound the covolume of $\Lambda'$:
$$\mu(H/\Lambda')  \le
c_0^{d\cdot\dm(\G)/2}{c'_0}^{d\cdot\frac12s}\left(\prod_{i=1}^{r}\frac{m_i!}{(2\pi)^{m_i+1}}\right)^{d}
\!\!\!\lambda_{v_0}\!\!\! \prod_{v\in
V_f}\frac{1}{(1- q_v^{-(m_1+1)})\ldots(1-q_v^{-(m_r+1)})},
$$
\begin{gather*}
\lambda_{v_0} = \frac{(1-q_{v_0}^{-(m_1+1)})\ldots(1-q_{v_0}^{-(m_r+1)})
q_{v_0}^{(\dm(\Mvo)+\dm(\cMvo))/2}}{\#\Mvo (\F_{v_0})}.
\end{gather*}
The $\lambda_{v_0}$-factor corresponds to the distinguished place $v_0$ of $k$
at which we have no control over the structure of $\G$. Still it is easy to
see that
\begin{equation*}
\lambda_{v_0} \le q_{v_0}^{(\dm(\Mvo)+\dm(\cMvo))/2} \le q_{v_0}^{\dm(\G)} \le p_0^{d\cdot\dm(\G)}
\end{equation*}
(here we use the assumption that $v_0$ lies over a fixed prime $p_0$).

Now, the Euler product
\begin{align*}
\prod_{v\in V_f} & \frac{1}{(1-q_v^{-(m_1+1)})\ldots(1-q_v^{-(m_r+1)})}\\
& = \zeta_k(m_1+1)\ldots \zeta_k(m_r+1)\\
& \le \zeta(m_1+1)^d\ldots \zeta(m_r+1)^d \le \zeta(2)^{d r} = \left(\frac{\pi^2}{6}\right)^{d r},
\end{align*}
where $\zeta_k$ is the Dedekind zeta function of $k$ and $\zeta$ is the Riemann
zeta function. The inequalities $\zeta_k(s)\le\zeta(s)^d$ and
$\zeta(s)\le\zeta(2)$ ($s\ge 2$) which we use here are elementary and easy to
check.

We obtain
\begin{align*}
\mu(H/\Lambda') &\le c_1^{d},\\
\text{where } c_1 &= c_0^{\frac12\dm(\G)}{c'_0}^{\frac12s}
\prod_{i=1}^{r}\frac{m_i!}{(2\pi)^{m_i+1}}\: p_0^{\dm(\G)}
\left(\frac{\pi^2}{6}\right)^{r}.
\end{align*}

\begin{rem} Instead of bounding the Euler product $\E(\Pa)$, one can also give its precise
expression (at least up to a rational factor which in our case is
$\lambda_{v_0}$) as a product of Dedekind zeta functions and certain
Dirichlet $L$-functions evaluated at $m_i+1$, $i\le r$. However, in order to
determine the $L$-factors a case-by-case analysis is needed. Since the bound we
get is sufficient for our purpose, we shall not go into details and skip the
case-by-case routine.
\end{rem}

\medskip

Now fix a prime $p'\neq p_0$ and look at the $p'$-congruence subgroup $\Lambda(p')$ of $\Lambda$.
The group $Q = \Lambda/\Lambda(p')$ is a quasi-semisimple
finite group of order at most $p'^{\:d\:\dm(\G)}$, and it contains an elementary
abelian $p'$-group $A$ of dimension at least $d$. This $A$ has at least
$p'^{[\frac14 d^2]}$ subgroups (\cite[Prop.~1.5.2]{LS}), hence $\Lambda$ has
at least $p'^{[\frac14 d^2]}$ subgroups of index at
most $c_3^d$, where $c_3 = p'^{\:\dm(\G)}$. This gives $p'^{[\frac14 d^2]}$ lattices
in $H$ of covolume at most $(c_1c_3)^d$.

We finally claim that any given lattice in this set of $p'^{[\frac14 d^2]}$
lattices has at most $p'^{\:c_4d}$ lattices within the set which are conjugate
to it in $H$. This indeed follows from Corollary~\ref{cor_s3}. Thus we get
$p'^{[\frac14 d^2 - c_4d]}\ge c_2^{d^2}$ different conjugacy classes of lattices in $H$ of
covolume at most $(c_1c_3)^d$ as promised.

\section{Proof of the upper bound} \label{section:upper bounds}\label{sec:uniform upper bnd}


Let us recall the main result of \cite{B} which we are going to use in this section (see also \cite{BGLS} for the groups of type $\An_1$):

\begin{thm} \label{thm:B}
Let $H$ be a semisimple Lie group of real rank $\ge2$ without compact factors. Denote by $\muu$
the number of conjugacy classes of maximal irreducible lattices in $H$ of covolume at most $x$. Then
for every $\epsilon > 0$ there exists $c = c(\epsilon, H)$ such
that $\muu \le x^{c(\log x)^\epsilon}$ for every $x\gg 0$.
\end{thm}

It is actually conjectured in \cite{B} that $\muu$ is polynomially bounded, but we will not need this conjecture here.

\medskip

We shall count the lattices of covolume at most $x$ by first counting the maximal ones (the number of which is small by Theorem~\ref{thm:B}),
and then counting finite index subgroups within such maximal lattices.

\medskip

In the proof below the following proposition will be used several times.

\begin{prop}\label{prop_82}\cite[Prop. 1.3.2(i)]{LS}
Let $G$ be a group, $N\lhd G$ and $Q = G/N$. Then
$$s_n(G) \le s_n(N)s_n(Q)n^{\rk(Q)},$$
where $s_n(X)$ denotes the number of subgroups of $X$ of index at most $n$ and $\rk(Q)$ is the rank of $Q$.
\end{prop}

It is easy to see that our results are independent of the
choice of the Haar measure $\mu$ on $H$. For the sake of convenience in this section we shall fix $\mu$
so that $\mu(H/\Gamma) \ge 1$ for every lattice $\Gamma$. This is possible
since by Kazhdan-Margulis theorem (see \cite[Chapter XI]{Rag}) there exists a positive lower bound for the
covolumes of lattices in $H$.

\medskip

We have:
\begin{equation}\label{eq_63}
\ru \le \muu \cdot \sup_{\substack{\Gamma \\ \mu(H/\Gamma)\le x}} \!\!\!\! s_x(\Gamma),
\end{equation}
where $\Gamma$ runs over the maximal lattices in $H$.

Every such maximal $\Gamma$ is equal to $N_H(\Lambda')$ as in Proposition~\ref{cor:B}, where $\Lambda' = \phi(\Lambda)$ in the notations there.

We can first use Proposition \ref{prop_82} to deduce that $s_x(\Gamma) \le s_x(\Lambda')s_x(Q)x^{\rk(Q)}$, where $Q = \Gamma/\Lambda'$.
By Proposition \ref{cor:B}, $|Q|\le x^{c_2}$ hence $s_x(Q) \le |Q|^{\log |Q|} \le x^{c_2^2\log x}$ and $\rk(Q)\le c_2\log x$.
Thus to prove the upper bound of Theorem \ref{theorem} it suffices to give a similar bound for $s_x(\Lambda')$. Clearly, $s_x(\Lambda')\le s_x(\Lambda)$. So it is sufficient to bound $s_x(\Lambda) = s_x(\hat{\Lambda})$, where $\hat{\Lambda}$ is the profinite completion of $\Lambda$.

To estimate $s_x(\hat{\Lambda})$ let us recall that we assume Serre's conjecture, i.e. that $\Lambda$ satisfies the congruence subgroup property. It means that the congruence kernel $C = \Ker(\hat{\Lambda} \to \prod_v\Pa_v)$ is finite. To simplify the exposition we shall assume that $C = \{e\}$ and later explain how to remove this assumption.

So we need to bound from above $s_x(\prod_v \Pa_v)$. Let $T$ be the set of ``bad'' valuations from Proposition \ref{cor1:B}. Thus, $\prod_{v\in T}q_v \le x^{c_1}$ (see there) and for every $v\not\in T$, $\Pa_v$ is hyperspecial. Now we can use Proposition \ref{prop_82} again, this time with $G = \prod_v{\Pa_v}$, $N = \prod_{v\not\in T}{\Pa_v}$ and $Q = \prod_{v\in T}\Pa_v$, to deduce
\begin{equation}\label{eq_8star}
s_x(\Lambda) \le s_x(N) s_x(Q) x^{\rk(Q)}.
\end{equation}

The group $Q$ is a product of $\#T$ $p$-adic analytic compact groups (for possibly different primes $p$). Collecting together those with the same $p$ (i.e. those $v\in T$ which lie over the same rational prime $p$), we get a subgroup $A_p$ such that $A_p \subseteq \prod_{v|p}\SL_s(\cO_v)$, where $\cO$ is the ring of integers of $k$, the field of definition of $\Lambda$, and $s$ is a fixed number such that $H \subseteq \SL_s(\R)$. If $d = d_k$ and $M_v$ is the maximal ideal of $\cO_v$, then we have $p^d = \prod_{v|p}[\cO_v:M_v]^{e_v} = \prod_{v|p} q_v^{e_v} = \prod_{v|p} p^{f_ve_v}$, where $q_v = p^{f_v}$ and $e_v$ denotes the ramification degree. Now, $\SL_s(\cO_v)$ is a $p$-adic analytic virtually pro-$p$ group of dimension $\le s^2f_ve_v$. It follows that $K_p = \Ker(A_p \to \prod_{v|p} \SL_s(\F_{q_v}))$ is a pro-$p$ group of rank at most $s^2d = O(\log x)$ (by \cite[Theorem 5.2 and Theorem 3.8]{DDMS}).
We can bound the rank of $A_p/K_p$ using the following result:
\begin{prop}\label{prop_83} \cite[Cor. 24, p. 326]{LS}
$$\rk(\GL_s(\F_{p^f})) < 2s^2f.$$
\end{prop}

Putting all this together, $Q$ is a product $\prod_{p\in S} A_p$ of finitely many groups $A_p$, where $S$ is the set of rational primes lying below $T$. It has a normal subgroup $K = \prod_{p\in S}K_p$ with $\rk_l(K)\le s^2d = O(\log x)$ for every $l$.
The quotient $Q/K$ is a subgroup of $\prod_{v\in T} \SL_s(\F_{q_v})$ and by Proposition \ref{prop_83},
$$ \rk(\prod_{v\in T} \SL_s(\F_{q_v})) \le \sum_{v\in T} \rk(\SL_s(\F_{q_v})) \le 2s^2\sum_{v\in T} f_v = O(\log x).$$
Here the last estimate follows from the fact that $\prod_{v\in T} p_v^{f_v} = \prod_{v\in T}q_v \le x^c$ (by Proposition \ref{cor1:B}). Thus, by definition of the rank, $\rk(Q/K) = O(\log x)$.

So, we deduce that $\rk(Q) = O(\log x)$.

\medskip

We are left by (\ref{eq_8star}) with bounding $s_x(Q)$ and $s_x(N)$.

\medskip

Let us consider $s_x(Q)$. Recall first:
\begin{prop}\label{prop_84} {\rm (See \cite[Cor. 1.7.5, p. 28]{LS})}
Let $X$ be a finite group. Then
$$ s_n(X) \le n^{\nu(n)+r+1}, $$
where $r = \rk(X)$ and $\nu(n)$ is the number of distinct prime divisors of $n$ (so $\nu(n) = O(\frac{\log n}{\log\log n})$ by the prime number theorem).
\end{prop}
We can apply Proposition \ref{prop_84} to the profinite group $Q$ to deduce that $s_x(Q)\le x^{c_3\log x}$ as needed. Before moving to bounding $s_x(N)$, let us summarize what we have seen so far as we will use it again later.
\begin{claim} \label{claim:7.5}
With the notations as above, let $T_0$ be a finite set of valuations of $k$ with  $\prod_{v\in T_0} q_v \le x^c$ and let $Q = \prod_{v\in T_0} \Pa_v$. Then $\rk(Q)=O(\log x)$ and $s_x(Q) \le x^{O(\log x)}$.
\end{claim}

Now we can turn to the more challenging task of bounding $s_x(N)$. This time $N$ is a product of infinitely many groups $N = \prod_{v\not\in T}\Pa_v$. Each group $\Pa_v$ is an extension of a pro-$p$ group $K_v^o$ by an almost simple group $L_v^o = \Pa_v/K_v^o$ of the form $\G_v(\F_{q_v})$, where $\G_v$ is the group scheme over the ring $\cO_v$ in $k_v$ (see \S\ref{sec:BT}).

Let $K^o = \prod_{v\not\in T} K_v^o$ and $L^o = \prod_{v\not\in T} L_v^o$, so $N/K^o \cong L^o$. Let $K_v$ be the preimage in $\Pa_v$ of the center of $L_v^o$ and $K = \prod_{v\not\in T}K_v$.

Following \cite[Window 3, \S2]{LS}, we say that a profinite group {\em $X$ is in $B_k$} (for a fixed $k\in\N$) if no composition factor of $X$ is isomorphic to $\mathrm{Alt}(m)$ with $m>k$ or to a classical finite simple group of Lie type of degree exceeding $k$ (here {\em degree} means the degree of the natural projective representation). Our group $N$, as well as any open subgroup of it, is in $B_k$ for a suitable $k$ depending only on the Lie group $H$ but not on $\Lambda$.

Recall that a {\em chief factor} of a group $X$ is a quotient $A/B$ where $B\subset A$ are both normal subgroups of $X$ and $B$ is a proper subgroup of $A$, maximal with respect to being normal in $X$. In this case $A/B$ is isomorphic to $S^m$ for some finite simple group $S$, and we say that this is a {\em non-abelian chief factor} if $S$ is non-abelian. We will say that $X$ {\em has simple non-abelian chief factors} if for every non-abelian chief factor $A/B$ as above we have $m = 1$. By Jordan-Holder theorem, the groups appearing as chief factors are determined by any chosen chief series. Thus if $X$ has a normal prosolvable subgroup $K$ with $X/K$ isomorphic to a direct product of non-abelian finite simple groups, one can deduce that $X$ has simple non-abelian chief factors. This is clearly the case for our group $N$.

Recall that a subgroup $M$ of $X$ is called {\em subnormal} if there exists a sequence $X = M_0 > M_1 > \ldots > M_m = M$ with $M_{i+1}\lhd M_i$ and $M_{i+1}$ is a maximal normal subgroup of $M_i$, in which case we say that $M$ is a subnormal subgroup of {\em length} $m$ in $X$. The number of non-abelian factors $M_i/M_{i+1}$ will be called the {\em non-abelian length} of $M$ in $X$.
\begin{lemma}\label{lemma_85}
Let $X$ be a profinite group with simple non-abelian chief factors and $M$ a subnormal subgroup of $X$ of non-abelian length $m_0$ in $X$. Let $C(M)$ be the core of $M$, i.e. $C(M) = \cap_{g\in X}M^g$ is the largest subgroup of $M$ which is normal in $X$. Then the non-abelian length of $C(M)$ in $X$ is equal to $m_0$.
\end{lemma}
\begin{proof}
Clearly, the non-abelian length of $C(M)$ is at least $m_0$. We prove the converse by induction on $m$ (the length of $M$ in $X$). Assume by the induction hypothesis that the non-abelian length of $C(M_{m-1})$ is $m_1$ and it is at most $m'_0$, which is the non-abelian length of $M_{m-1}$ in $X$.

If $M_{m-1}/M$ is abelian, then $m_0 = m'_0$. The group $\bar{M} = M_m\cap C(M_{m-1})$ is a normal subgroup of $C(M_{m-1})$ (since $M_m\lhd M_{m-1}$ and $C(M_{m-1})\lhd M_{m-1}$) and
$$C(M_{m-1})/\bar{M} \cong C(M_{m-1})M_m / M_m \unlhd M_{m-1}/M_m,$$
so it is abelian. It follows that $C(M_{m-1})/C(M_m)$ is also abelian (since $C(M_m) = C(\bar{M})$ and $C(\bar{M})$ is the intersection of the $X$-conjugates of $\bar{M}$ in $C(M_{m-1})$, where the latter is normal in $X$). Thus the non-abelian length of $C(M)$ is also $m_1$ and we are done with this case.

If $M_{m-1}/M$ is non-abelian, then $m_0 = m_0' + 1$. In the notations of the previous paragraph we get that $C(M_{m-1})/\bar{M}$ is isomorphic to a normal subgroup of a simple group $S = M_{m-1}/M_m$. If it is trivial, then $C(M) = C(M_{m-1})$ and we finish by induction. If not, then $C(M_{m-1})/C(M_m)$ is a product of copies of $S$ on which $X$ acts transitively. But as every non-abelian chief factor of $X$ is simple, there is only one such copy and the non-abelian length of $C(M)$ in $X$ is at most $m_0' + 1 = m_0$.
\end{proof}

Assume now further that $X$ is in $B_k$ and recall an important result of Babai, Cameron and P\'alfy (cf. \cite[Theorem 4, p. 339]{LS}):
\begin{thm}\label{thm:BCP}
Let $Y$ be a primitive permutation group of degree $n$ and $Y\in B_k$. Then $|Y|\le n^{f_1(k)}$, where $f_1(k)$ depends only on $k$.
\end{thm}

We mention in passing that while Theorem \ref{thm:BCP} as stated depends on the classification of the finite simple groups (CFSG), the way we are going to use it here (for profinite groups with ``known'' finite simple factors) is independent of the CFSG.

The important corollary for us is the following:
\begin{prop}\label{prop_87}
If $X$ is in $B_k$ and $D$ is a subgroup of $X$ of index $n$ then there exists a subnormal subgroup $M$ of $X$ contained in $D$ with $[X:M]\le n^{f_1(k)}$.
\end{prop}

\begin{proof}
Let $X  = D_0 > D_1 > \ldots > D_d = D$ be a sequence of subgroups such that $D_{i+1}$ is a maximal subgroup of $D_i$. Define by induction $M_{i+1}$ to be the core of $D_{i+1}\cap M_i$ in $M_i$ (i.e. the maximal normal subgroup of $M_i$ contained in $D_{i+1}\cap M_i$). Note that either $D_{i+1}\cap M_i = M_i$ (in which case $D_{i+1} \supseteq M_i$ and $M_{i+1} = M_i$) or $D_{i+1}\cap M_i$ is a maximal subgroup of $M_i$ of index at most $[D_i:D_{i+1}]$. The action of $M_i$ on the coset space $M_i / (D_{i+1}\cap M_i)$ is by a primitive permutation group, which is in $B_k$ by our assumption. Thus, Theorem \ref{thm:BCP} implies that $|M_i/M_{i+1}| \le [D_i:D_{i+1}]^{f_1(k)}$, and altogether $|X/M_d| \le [X:D]^{f_1(k)}$.
\end{proof}

Let us now apply all these preparations to the group $N = \prod_{v\not\in T} \Pa_v$. For this group we have an extra property:
\begin{lemma}\label{lemma_88}
Let $M$ be a subnormal subgroup of $N$ with a sequence $N = M_0 \rhd M_1 \rhd \ldots \rhd M_m = M$ in which $M_i/M_{i+1}$ are finite simple groups. Then for every $i$ for which $M_i/M_{i+1}$ is non-abelian, there is a unique $v$ such that $M_i\cap\Pa_v = \Pa_v$ while $M_{i+1}\cap\Pa_v = K_v$. Here $K_v$ is the unique prosolvable subgroup of $\Pa_v$ for which $\Pa_v/K_v$ is a non-abelian finite simple group (such $K_v$ exists since $\Pa_v$ is hyperspecial).
\end{lemma}
\begin{proof}
First note that $M_i\cap \Pa_v$ is a subnormal subgroup of $\Pa_v$ and $K_v$ is the unique maximal normal subgroup of $\Pa_v$ (see \S\ref{sec:BT}).
For almost every $v$, $M\supseteq\Pa_v$, but for finitely many $v$ this inclusion may not hold, in which case there is a first $i$ such that $M_i\supseteq\Pa_v$, so $M_i\cap\Pa_v = \Pa_v$, but $M_{i+1}\cap\Pa_v \subseteq K_v$. In this case
$$\Pa_v / (M_{i+1}\cap\Pa_v) = M_{i+1}\Pa_v / M_{i+1} \lhd M_i/M_{i+1},$$
and so $\Pa_v / (M_{i+1}\cap\Pa_v) = M_i / M_{i+1}$.

For a given $i$, there is only one such $v$. Indeed, assume $M_i \supseteq \Pa_{v_1}\times\Pa_{v_2}$ but $M_{i+1}\cap\Pa_{v_1}\subseteq K_{v_1}$ and $M_{i+1}\cap\Pa_{v_2}\subseteq K_{v_2}$.  Now,  $M_{i+1}\cap (\Pa_{v_1}\times\Pa_{v_2})$ is a normal subgroup of $\Pa_{v_1}\times\Pa_{v_2}$; looking at it modulo $K_{v_1}\times K_{v_2}$ we get a normal subgroup of the product $\Pa_{v_1}/K_{v_1} \times \Pa_{v_2}/K_{v_2}$ of two non-abelian finite simple groups, which has a trivial intersection with each factor. It is therefore the trivial subgroup, i.e.,
$M_{i+1}\cap(\Pa_{v_1}\times\Pa_{v_2}) \subseteq K_{v_1}\times K_{v_2}$. So
$$ (\Pa_{v_1}\times\Pa_{v_2}) / (M_{i+1}\cap(\Pa_{v_1}\times\Pa_{v_2})) \cong M_{i+1}(\Pa_{v_1}\times\Pa_{v_2})/M_{i+1}.$$
The right hand side is a subnormal subgroup of $M_i/M_{i+1}$ which is a simple group but the left hand side has a quotient $(\Pa_{v_1}\times\Pa_{v_2}) / (K_{v_1}\times K_{v_2})$ which is a product of two simple groups --- a contradiction.

Finally, we note that since all the non-abelian composition factors of $X$ are obtained from the various $\Pa_v/K_v$, it is clear that for every $i$ there is such a place $v$.
\end{proof}

Note that if $E$ is a normal subgroup of $N$, then for every $v\not\in T$, either $E\cap\Pa_v\supseteq\Pa_v$ or $E\cap\Pa_v\subseteq K_v$, in which case $L_v = \Pa_v/K_v$ is one of the non-abelian composition factors appearing in $N/E$.

Before continuing, let us make an observation which will be needed later.
\begin{cor}\label{cor:7.10}
Let $d(N)$ denote the minimal number of generators of the profinite group $N = \prod_{v\not\in T}\Pa_v$. Then $d(N)=O(\log x)$.
\end{cor}
\begin{proof}
Indeed, $K^o$ is a product of infinitely many $p$-adic analytic pro-p groups ${\mathcal K}_p = \prod_{v|p}K_v^o$, but for every $p$, ${\mathcal K}_p$ is a subgroup of a uniform pro-p group of dimension bounded by $O(\log x)$ and hence $d(K^o)=O(\log x)$. The quotient $N/K^o$ is an infinite product of finite quasi-simple groups. The multiplicity of each one is bounded by $O(\log x)$ and hence $d(N/K^o)=O(\log x)$. Altogether $d(N)$ is also bounded by a constant multiple of $\log x$.
\end{proof}

We are now ready to bound $s_x(N)$: If $D$ is a subgroup of index at most $x$ in $N$, then by Proposition \ref{prop_87} it contains a subnormal subgroup $M$ of $N$ of index at most $x^c$. The non-abelian composition factors between $M$ and $N$ correspond to a finite set $T_1$ of valuations $v\not\in T$, and since $q_v \le |L_v|\le q_v^{\dm(\G)}$, it follows that $\prod_{v\in T_1}q_v \le x^{c_1}$.
Let $C(M)$ be the core of $M$. By Lemma \ref{lemma_85} it has the same non-abelian finite simple composition factors. Moreover, from the discussion above it follows that $C(M)$ contains $\Pa_v$ for every $v\not\in T\cup T_1$.

Now note that the number of possibilities for $T_1$ is bounded by $x^{c_2}$ (this follows from $\prod_{v\in T_1}q_v \le x^{c_1}$ and \cite[Sec. 4.1 and Prop. 3.2(ii)]{B}), and so we can fix $T_1$ and reduce the problem to estimating $s_x(\prod_{v\in T_1} \Pa_v)$. This brings us to the situation which was already considered in this section.
The required estimate is provided by Claim \ref{claim:7.5}. This finishes the proof of the upper bound of Theorem \ref{theorem} under the assumption that $C = \Ker(\hat{\Lambda} \to \prod\Pa_v)$ is trivial.

\medskip

Let us now explain how to handle the case when $C$ is non-trivial.
First recall that the CSP and MP imply that it is always cyclic --- a subgroup of $\mu(k)$ --- the group of roots of unity in $k$ (cf.~\cite[Theorem~2]{PR10}). Recall that if $\mu_n$ is the group of $n$-roots of unity, then $\Q[\mu_n]/\Q$ is an extension of degree $\phi(n)$ which is at least $\sqrt{n}$. This implies that $\mu(k)$ is of order bounded by $O(d_k^2) = O(\log^2 x)$. Thus $C$ is of order $O(\log^2 x)$.
Secondly, note that along the way of the proof we saw that every subgroup $M$ of $N$ of index at most $x$ contains an infinite product $N_1 = \prod_{v\not\in T_1} \Pa_v$, where $T_1$ satisfies $\prod_{v\in T_1} q_v \le x^c$. Therefore, the number of generators $d(M) \le d(N_1) + \rk(Q_1)$, where $Q_1 = N/N_1 = \prod_{v\in T_1}\Pa_v$. By Corollary \ref{cor:7.10}, $d(N_1)$ is bounded by $O(\log x)$, and by Claim \ref{claim:7.5}, $\rk(Q_1)$ is bounded by $O(\log x)$. Hence, $d(M) = O(\log x)$.
Moreover, using again Claim \ref{claim:7.5} we deduce that every subgroup $M$ of $\hat{\Lambda}/C$ of index at most $x$ can be generated by at most $c'\log x$ elements. We can now apply Lemma~1.3.1(i) from \cite[p.~15]{LS}: As $C$
acts trivially on the group $\hat{\Lambda}$, derivations are just homomorphisms and it follows
that the number of subgroups of $\hat{\Lambda}$ whose projection in $\hat{\Lambda}/C$ is $M$ is bounded by $|C|^{d(M)} \le (\log x)^{c'\log x}$. This  finishes the proof of Theorem \ref{theorem}. \qed

\section{Growth of lattices in semisimple Lie groups}\label{section:semisimple}\label{sec:semi thm1}

In this final section we are going to discuss how to extend the results of the paper to semisimple Lie groups. Given such a group $H$ it is natural to consider only {\em irreducible} lattices in $H$, so from now on $\mathrm{L}_H(x)$ denotes the number of conjugacy classes of irreducible lattices in $H$ of covolume at most $x$.
We recall (see \S\ref{sec:ar}) that $H$ contains irreducible lattices only if it is isotypic, and that we can assume that $H = (\prod_{j=1}^{a}\G_j(\R) \times \G_{a+1}(\C)^{b})^o$ for some absolutely simple $\R$-groups $\G_j$, $j = 1, \dots, a+1$.

To obtain an analogue of the lower bound of Theorem \ref{theorem} which was proved in \S\ref{sec:thm1 lower} for a simple group $H$, we need to modify the choice of the fields of definition $(k_i)$: now the fields have to be chosen so that
$$ r_1(k_i) \ge a,\ r_2(k_i) = b \text{ and } \D_{k_i}^{1/d_i}\le c_0.$$
This can be always achieved using Corollary \ref{cor:Pisot}.

For each of the fields $k = k_i$ we have to define an extension $l$ as in \S\ref{sec:thm1 lower}. Let $\G$ be an admissible group (in the sense of \S\ref{sec:ar}) defined over $k$, and suppose that $\G$ is inner over $k_{v_j}$ for some $t_1$ real places $v_j$ of $k$ and is outer over the remaining $t_2 = r_1(k) - t_1$ real places. We note that either $t_1$ or $t_2$ depend only on the Lie group $H$ (the former is the case when the compact real form of the simply connected covers of the simple factors of $H$ is outer, and the latter, if it is inner). If $t_2 = 0$ (i.e. $\G$ is an inner form over $k$), we let $l = k$. Otherwise, $l$ is defined as a quadratic extension of $k$ such that precisely $t_1$ real places of $k$ split in $l$. Similarly to \S\ref{sec:thm1 lower}, we can always achieve that $\D_{l/k} \le {c'_0}^{d}$: If $k$ is a totally real field, we can take $l = k[\sqrt{-(1-\theta_1)\ldots(1-\theta_{t_1})}]$ or
$l = k[\sqrt{(1-\theta_1)\ldots(1-\theta_{t_2})}]$ depending on the above mentioned two cases, where $\theta_1$, \ldots, $\theta_t$ are Pisot numbers in $k$ provided by Lemma \ref{lemma:Pisot}(b). If $k$ has complex places and $t_2\neq 0$, we can first consider its maximal totally real subfield $k'$, using Pisot numbers define its quadratic extension $l'$ which splits $t_1$ infinite places of $k'$ (which correspond to real places of $k$ in the extension $k/k'$) and has $\D_{l'/k'}\le {c_1'}^{d(k')}$, and then define $l$ as a compositum of $k$ and $l'$.

With such fields $k$ and $l$ at hand we can repeat the rest of the argument in \S\ref{sec:thm1 lower} and thus show that the lower bound in Theorem \ref{theorem} is valid for any semisimple group $H$ which contains irreducible lattices (i.e. for any isotypic semisimple Lie group).

The proof of the upper bound in \S\ref{sec:uniform upper bnd} does not use the assumption that the Lie group is simple and can be applied without any changes to semisimple groups $H$ assuming validity of the congruence subgroup property and Margulis-Platonov conjecture.

Thus we obtain the following generalization of Theorem \ref{theorem} to semisimple Lie groups.
\begin{thm} \label{thm:semi}
Let $H$ be an isotypic semisimple Lie group of real rank at least $2$. Then
\begin{itemize}
\item[(i)] There exists a positive constant $a$ such that $\ru \ge x^{a\log
x}$ for all sufficiently large $x$.
\item[(ii)] Assuming the CSP and MP, there exists a positive constant $b$ such that $\ru \le x^{b\log x}$ for
all sufficiently large $x$.
\end{itemize}
\end{thm}

\end{document}